\documentclass[12pt]{article}
\usepackage[margin=0.75in,footskip=0.25in]{geometry}
\usepackage{latexsym}
\usepackage{amssymb,amsmath,amsthm,accents}
\usepackage{amsfonts,array}
\usepackage{float}
\usepackage{color}
\usepackage{paralist}
\usepackage{graphicx}
\usepackage{caption,subcaption}
\usepackage{enumitem}
\usepackage{authblk}
\usepackage{setspace}
\usepackage{etoolbox}
\AtBeginEnvironment{align*}{\let\oldfrac\frac\def\frac#1#2{\mathchoice
{\tfrac{#1}{#2}}%
{\oldfrac{#1}{#2}}%
{\oldfrac{#1}{#2}}%
{\oldfrac{#1}{#2}}}
}
\usepackage{tikz}
\usetikzlibrary{arrows,automata}
\usepackage{mathtools}
\DeclarePairedDelimiter\abs{\lvert}{\rvert}

\DeclarePairedDelimiter\ceil{\lceil}{\rceil}
\DeclarePairedDelimiter\floor{\lfloor}{\rfloor}
\newtheorem{theorem}{Theorem}[section]

\newtheorem{cor}[theorem]{Corollary}
\newtheorem{prop}[theorem]{Proposition}

\newtheorem{con}[theorem]{Conjecture}
\newcommand{\n}{\noindent}

\begin{document}
\title{\bf DNA graph characterization for the line digraph of dicycle with $\boldsymbol{\floor*{\frac{n}{3}}}$ chords, $\boldsymbol{\infty}$-digraph $\boldsymbol{C_n \cdot C_p}$, and 3-blade-propeller $\boldsymbol{C_n \cdot C_p \cdot C_q}$}

\author{\small Inne Singgih}
\affil{\small University of South Carolina\\{\tt \small  isinggih@email.sc.edu}}

\date{}
\maketitle
\begin{abstract}
\noindent DNA graph has important contribution in completing the computational step of DNA sequencing process. Using $(\alpha,k)$-labeling, several families of digraphs have characterized as DNA graphs. Dicycles and dipaths are DNA graphs, rooted trees and self adjoint digraphs are DNA graphs if and only if their maximum degree is not greater than four, while the $m\textsuperscript{th}$ line digraph of dicycle with one chord is a DNA graph for all $m\in\mathbb{Z}^+$. In this paper we construct $(\alpha,k)$-labeling to show that for all $m\in\mathbb{Z}^+$, the $m\textsuperscript{th}$ line digraph of dicycle $C_n$ with $\floor*{\frac{n}{3}}$ chords are DNA graphs for $n\geq 6$, and the $m\textsuperscript{th}$ line digraph of $\infty$-digraph $C_n\cdot C_p$ and 3-blade-propeller $C_n\cdot C_p \cdot C_q$ are DNA graphs for $n\geq 3$ and certain values of $p$ and $q$.\\

\end{abstract}
\section{Introduction}

\n Given a digraph $D=(V,A)$, for every arc $a=uv\in A$, $u$ is called the \textit{tail} of $a$and $v$ is called the \textit{head} of $a$. The \textit{line digraph} of $D$, $L(D)$, is a digraph with vertex set $V(L(D))=A(D)$ and an arc $xy$ exists in $A(L(D))$ iff the head of $x$ is the tail of $y$ in $D$. A digraph $D$ is said to be \textit{self-adjoint} if $D\cong L(D)$ \cite{baca}. Line Digraph of $L(D)$ denoted by $L(L(D))$ or simply $L^2(D)$, similarly line digraph of $L^m(D)$ denoted by $L^{m+1}(D)$ for $m\in\mathbb{N}$ \cite{singgih}. \\[-2mm]

\n Let $\alpha>0$ and $k>1$ be integers. A digraph $D=(V,A)$ is said to be $(\alpha,k)$-labeled if it is possible to label each vertex $x$ of $D$ with a $k$-length label $\left(l_1(x),l_2(x),\ldots,l_k(x)\right)$, such that 
\begin{enumerate}[nolistsep]
\item[(1)] $l_i(x)\in \{1,2,\ldots,\alpha\}$ for all $x\in V$ and $i\in \{1,2,\ldots,k\}$,
\item[(2)] each vertex has different labels: $\left(l_1(x),l_2(x),\ldots,l_k(x)\right) \neq \left(l_1(y),l_2(y),\ldots,l_k(y)\right)$ if $x\neq y$,
\item[(3)] deBruijn property holds: $xy\in A \Leftrightarrow l_i(x)=l_{i-1}(y)$ for $i\in \{2,3,\ldots,k\}$.
\end{enumerate} 
\n $D=(V,A)$ is said to be quasi-$(\alpha,k)$-labeled if everything above holds while property (3) is relaxed to $xy\in A \Rightarrow l_i(x)=l_{i-1}(y)$ for $i\in \{2,3,\ldots,k\}$. Obviously if a digraph $D$ is $(\alpha,k)$-labeled then $D$ is quasi-$(\alpha,k)$-labeled. \\[-2mm]

\n A \textit{DNA graph} is a digraph $D$ that is $(4,k)$-labeled for some positive integer $k$ \cite{li}. Here $\alpha=4$ translates the label $1,2,3,4$ into the four nucleotide bases: Adenine $(A)$, Cytosine $(C)$, Guanine $(G)$, and Thymine $(T)$, so that each vertex of $D$ represent a spectrum of nucleotide bases, that are to be merged into the target DNA strand (longer spectrum) in the DNA sequencing process.\\[-2mm]

\section{Known results}

\n The relationships between DNA sequencing and labeled graphs were described for the first time by Lysov et al. \cite{lysov} and Pevzner \cite{pevzner}.\\[-7mm]

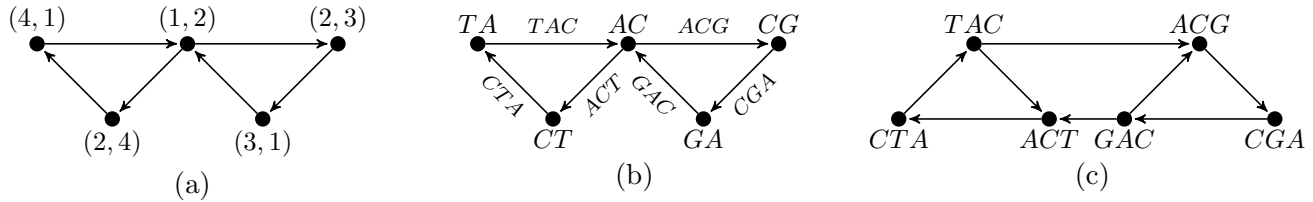
\begin{figure}[h]
\begin{subfigure}{.325\textwidth}
\centering
\begin{tikzpicture}[->,>=stealth',shorten >=1pt,auto, semithick]
\tikzstyle{every state}=[draw,circle,thick, fill=black, minimum size=5pt,
                            inner sep=0pt]
	\node[state] (41) at (0,1) [label={[label distance=-1mm]above:{\footnotesize ${(4,1)}$}}]{ };
	\node[state] (12) at (2,1) [label={[label distance=-1mm]above:{\footnotesize ${(1,2)}$}}]{ };
	\node[state] (23) at (4,1) [label={[label distance=-1mm]above:{\footnotesize ${(2,3)}$}}]{ };
	\node[state] (24) at (1,0) [label={[label distance=-1mm]below:{\footnotesize ${(2,4)}$}}]{ };
	\node[state] (31) at (3,0) [label={[label distance=-1mm]below:{\footnotesize ${(3,1)}$}}]{ };
	
	\path (41)  edge node { }(12)   
 	      (12)  edge node { }(23) edge node { }(24)
 	      (23)  edge node { }(31)
 	      (31)  edge node { }(12)
 	      (24)  edge node { }(41) ;
\end{tikzpicture}    
\vspace{-2mm}
\caption{ }
\end{subfigure}
\begin{subfigure}{.325\textwidth}
\centering
\begin{tikzpicture}[->,>=stealth',shorten >=1pt,auto, semithick]
\tikzstyle{every state}=[draw,circle,thick, fill=black, minimum size=5pt,
                            inner sep=0pt]
	\node[state] (TA) at (0,1) [label={[label distance=-1mm]above:{\footnotesize ${TA}$}}]{ };
	\node[state] (AC) at (2,1) [label={[label distance=-1mm]above:{\footnotesize ${AC}$}}]{ };
	\node[state] (CG) at (4,1) [label={[label distance=-1mm]above:{\footnotesize ${CG}$}}]{ };
	\node[state] (CT) at (1,0) [label={[label distance=-1mm]below:{\footnotesize ${CT}$}}]{ };
	\node[state] (GA) at (3,0) [label={[label distance=-1mm]below:{\footnotesize ${GA}$}}]{ };
	
	\path (TA)  edge node {{\scriptsize ${TAC}$}}(AC)   
 	      (AC)  edge node {{\scriptsize ${ACG}$}}(CG) 
 	            edge node[rotate=45,xshift=-5mm] {{\scriptsize ${ACT}$}}(CT)
 	      (CG)  edge node[rotate=45,xshift=-5mm] {{\scriptsize ${CGA}$}}(GA)
 	      (GA)  edge node[rotate=-45,xshift=5mm] {{\scriptsize ${GAC}$}}(AC)
 	      (CT)  edge node[rotate=-45,xshift=5mm] {{\scriptsize ${CTA}$}}(TA) ;
\end{tikzpicture}    
\vspace{-2mm}
\caption{ }
\end{subfigure}
\begin{subfigure}{.35\textwidth}
\centering
\begin{tikzpicture}[->,>=stealth',shorten >=1pt,auto, semithick]
\tikzstyle{every state}=[draw,circle,thick, fill=black, minimum size=5pt,
                            inner sep=0pt]
	\node[state] (412) at (1,1) [label={[label distance=-1mm]above:{\footnotesize ${TAC}$}}]{ };
	\node[state] (123) at (4,1) [label={[label distance=-1mm]above:{\footnotesize ${ACG}$}}]{ };
	\node[state] (241) at (0,0) [label={[label distance=-1mm]below:{\footnotesize ${CTA}$}}]{ };
	\node[state] (124) at (2,0) [label={[label distance=-1mm]below:{\footnotesize ${ACT}$}}]{ };
	\node[state] (312) at (3,0) [label={[label distance=-1mm]below:{\footnotesize ${GAC}$}}]{ };
	\node[state] (231) at (5,0) [label={[label distance=-1mm]below:{\footnotesize ${CGA}$}}]{ };
	
	\path (412)  edge node { }(123) 
	             edge node { }(124)  
 	      (123)  edge node { }(231) 
 	      (231)  edge node { }(312)
 	      (312)  edge node { }(124) 
 	             edge node { }(123)
 	      (124)  edge node { }(241)
 	      (241)  edge node { }(412) ;
\end{tikzpicture}  
\vspace{-2mm}
\caption{ }  
\end{subfigure}
\vspace{-3mm}
\caption{\small (a) A $(2,4)$-labeling of a digraph and the corresponding (b) Pevzner graph and (c) Lysov graph.}
\label{fig: letternumber}
\end{figure}

\n The Pevzner graph (graph (b)) is obtained by translating $1,2,3,4$ in the vertices' label in graph (a) into $A, C, G, T$, respectively. The arcs' label is obtained by merging the label of its head and tail on the overlapping bases. The Lysov graph (graph (c)) is the line digraph of the Pevzner graph. An Eulerian path in graph (b) is equivalent to Hamiltonian path in graph (c). Both paths, if starting at the top left vertex, represent the target spectrum $TACGACTA$.\\[-2mm]

\begin{prop}
\emph{\cite{li}} If a digraph $D$ is quasi-$(\alpha-1,k)$-labeled, then $D$ is $(\alpha,k)$-labeled.
\end{prop}

\begin{theorem}
\emph{\cite{li}} If a digraph $D$ is quasi-$(\alpha,k-1)$-labeled, then $L(D)$ is $(\alpha,k)$-labeled.
\end{theorem}

\begin{cor}
If a digraph $D$ is quasi-$(\alpha,k-1)$-labeled for $\alpha\leq 4$ and $k>2$, then its $m\textsuperscript{th}$ line digraph $L^m(D)$ is a DNA graph for all positive integer $m$.
\label{cor: repeat}
\end{cor}

\begin{theorem}
\emph{\cite{li}} Every dipath and every dicycle are DNA graphs, while rooted trees and self-adjoint digraphs are DNA graphs when $\Delta\leq 4$.
\label{th: known}
\end{theorem} 

\n For results in Theorem \ref{th: known}, the line digraph of the stated digraphs are either themselves or smaller digraphs of the same family, so applying Corollary \ref{cor: repeat} to these results is not useful. In \cite{singgih}, we managed to construct a quasi-$(4,k)$-labeling for dicycle with one chord $C^{\floor*{\frac{n}{2}}}_n$, where $\floor*{\frac{n}{2}}$ is the distance between the head and the tail of the arc that serves as the chord. As $\abs*{L^{m+1}\left(C^{\floor*{\frac{n}{2}}}_n\right)}>\abs*{L^{m}\left(C^{\floor*{\frac{n}{2}}}_n\right)}$ for all $m\in\mathbb{N}$, applying Corollary \ref{cor: repeat} gives infinitely many new types of DNA graphs.

\begin{theorem}
\emph{\cite{singgih}} $L^m\left(C^{\floor*{\frac{n}{2}}}_n\right)$ is a DNA graph for all $n\geq 4$ and positive integer $m$.
\label{th: 1ch}
\end{theorem}

\n Finding the construction of an $(\alpha,k)$-labeling of random given digraphs is NP-complete \cite{NP}. Working on the line digraph using quasi-$(\alpha,k)$-labeling reduces the complexity since finding a Eulerian (instead of Hamiltonian) path can be done in polynomial time. However, there were no further results ever published for other basic families of non self-adjoint digraphs which line digraph is bigger than themselves, such as dicycle with multiple chords or ladder digraphs.\\[-2mm]

\n In Section 3, 4, and 5 we provides the construction of quasi-$(\alpha,k)$-labeling for dicycles with $\floor*{\frac{n}{3}}$ chords, $\infty$-digraph $C_n \cdot C_p$, and 3-blade-propeller $C_n \cdot C_p \cdot C_q$, respectively. Applying Corollary \ref{cor: repeat}, their $m\textsuperscript{th}$ line digraph are DNA graph for any positive integer $m$. In Section 4 we also show the relation between the line digraphs of a certain $\infty$-digraph to ladder graphs. For convenience, in the constructions we use the notation $\{f_i\}_{i=1}^n$ to denotes the set $\{f_1,f_2,\ldots,f_n\}$. We also omits the parentheses and commas on the vertex labels in all tables and figures to save space, e.g., use label $123$ instead of $(1,2,3)$.

\section{Dicycle with $\boldsymbol{\floor*{\frac{n}{3}}}$ chords}

\n In this paper we work on dicycle with $\floor*{\frac{n}{3}}$ chords $^*C_n$ with $V(^*C_n)=\{v_i\}_{i=i}^n$ and chords set $C=\{v_{i-2}v_{i}\}_{i=3}^{n-t}$ when $n\equiv t \bmod 3$. This way $\abs{C}=\floor*{\frac{n}{3}}$ and the distance between the head and the tail of each chord is 2.

\begin{theorem}
$L^m\left(^*C_n\right)$ is a DNA graph only for $4\leq n \leq 14$ and positive integer $m$.
\end{theorem}
\begin{proof}
We want the construction of a quasi-$(\alpha,k)$-labeling for $^*C_n$ where $\alpha\leq 4$. The theorem then follows from Corollary \ref{cor: repeat}. First observe that the label for the vertices located between the tail and the head of a chord, $\{v_2,v_5,\ldots,v_{n-t-1}\}$ where $n\equiv t \bmod 3$, must be a repetition of the same number. WLOG for $v_1,v_2,v_3$: suppose $l(v_1)=abA$ then $l(v_2)=bAc$ and $l(v_3)=Acd$ for some $a,b,c,d\in \{1,2,3,4\}$ and $A$ is a string of numbers in $\{1,2,3,4\}$ with any positive length. Since there is a chord from $v_1$ to $v_3$, we must have $bA=Ac$, which only true when $b=c=$ all numbers in $A$. Since $\alpha\leq 4$, $l(v_2)$ is one of the $k$-length label $(1,1,\ldots,1),(2,2,\ldots,2),(3,3,\ldots,3)$, or $(4,4,\ldots,4)$. Hence in general there can only be at most 4 vertices located in between the head and tail of a chord. Since all labels must be distinct, there are at most 4 chords. Since dicycle with 1 chord already discussed in \cite{singgih}, we are left to show valid constructions for $6\leq n \leq 14$.\\
\n One of the valid construction using $k=3$ is given in Table \ref{tab: chords}.
\end{proof}

\vspace{-3mm}
\begin{table}[h!]
\centering
\renewcommand{\arraystretch}{1.3}
\small
\begin{tabular}{|c|l|}
\hline
$n$ & $l(v_1),l(v_2),\ldots,l(v_n)$ \\ \hline
6   & 211,111,112,122,222,221       \\ \hline
7   & 311,111,112,122,222,223,231   \\ \hline
8   & 311,111,112,122,222,223,233,331  \\ \hline
9   & 311,111,112,122,222,223,233,333,331  \\ \hline
10  & 211,111,112,122,222,223,233,333,332,321 \\ \hline
11  & 211,111,112,122,222,223,233,333,332,322,221 \\ \hline
12  & 411,111,112,122,222,223,233,333,334,344,444,441  \\ \hline
13  & 211,111,112,122,222,223,233,333,334,344,444,442,421  \\ \hline
14  & 211,111,112,122,222,223,233,333,334,344,444,442,422,221 \\ \hline
\end{tabular}
\caption{\small Quasi-$(4,3)$-labelings of $^*C_n$ for $6\leq n \leq 14$.}
\label{tab: chords}
\end{table}

\vspace{-3mm}
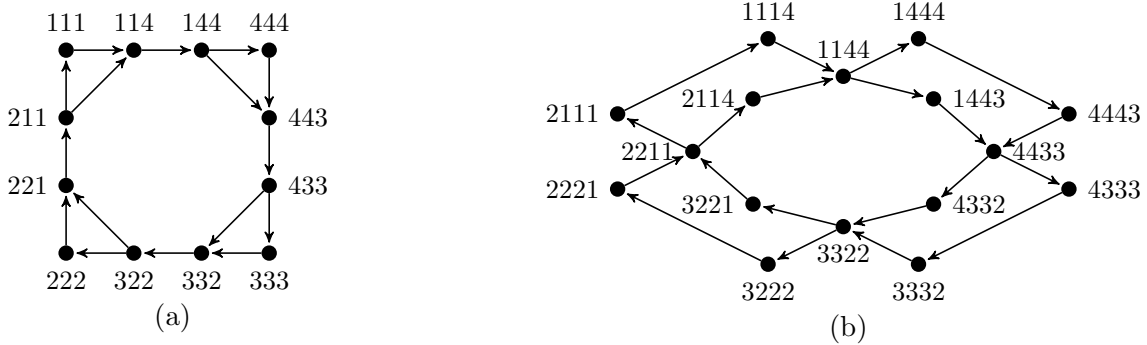
\begin{figure}[h]
\begin{subfigure}{.4\textwidth}
\centering
\begin{tikzpicture}[->,>=stealth',shorten >=1pt,auto, semithick,scale=0.9]
\tikzstyle{every state}=[draw,circle,thick, fill=black, minimum size=5pt,
                            inner sep=0pt]
	\node[state] (1)  at (0,2) [label={left:{\footnotesize ${211}$}}]{ };
	\node[state] (2)  at (0,3) [label={above:{\footnotesize ${111}$}}]{ };
	\node[state] (3)  at (1,3) [label={above:{\footnotesize ${114}$}}]{ };
	\node[state] (4)  at (2,3) [label={above:{\footnotesize ${144}$}}]{ };
	\node[state] (5)  at (3,3) [label={above:{\footnotesize ${444}$}}]{ };
	\node[state] (6)  at (3,2) [label={right:{\footnotesize ${443}$}}]{ };
	\node[state] (7)  at (3,1) [label={right:{\footnotesize ${433}$}}]{ };
	\node[state] (8)  at (3,0) [label={below:{\footnotesize ${333}$}}]{ };
	\node[state] (9)  at (2,0) [label={below:{\footnotesize ${332}$}}]{ };
	\node[state] (10) at (1,0) [label={below:{\footnotesize ${322}$}}]{ };
	\node[state] (11) at (0,0) [label={below:{\footnotesize ${222}$}}]{ };
	\node[state] (12) at (0,1) [label={left:{\footnotesize ${221}$}}]{ };
	\path (1)  edge node { }(2)  edge node { }(3)   
 	      (2)  edge node { }(3)
 	      (3)  edge node { }(4)   
 	      (4)  edge node { }(5)  edge node { }(6)
 	      (5)  edge node { }(6)   
 	      (6)  edge node { }(7)
 	      (7)  edge node { }(8)  edge node { }(9)  
 	      (8)  edge node { }(9)
 	      (9)  edge node { }(10)   
 	      (10) edge node { }(11) edge node { }(12)
 	      (11) edge node { }(12)
 	      (12) edge node { }(1) ;
\end{tikzpicture}    
\vspace{-2mm}
\caption{ }
\end{subfigure} 
\begin{subfigure}{.6\textwidth}
\centering
\begin{tikzpicture}[->,>=stealth',shorten >=1pt,auto, semithick]
\tikzstyle{every state}=[draw,circle,thick, fill=black, minimum size=5pt,
                            inner sep=0pt]
	\node[state] (12)   at (0,2)     [label={left:{\footnotesize ${2111}$}}]{ };
	\node[state] (23)   at (2,3)     [label={above:{\footnotesize ${1114}$}}]{ };
	\node[state] (34)   at (3,2.5)   [label={above:{\footnotesize ${1144}$}}]{ };
	\node[state] (45)   at (4,3)     [label={above:{\footnotesize ${1444}$}}]{ };
	\node[state] (56)   at (6,2)     [label={right:{\footnotesize ${4443}$}}]{ };
	\node[state] (67)   at (5,1.5)   [label={right:{\footnotesize ${4433}$}}]{ };
	\node[state] (78)   at (6,1)     [label={right:{\footnotesize ${4333}$}}]{ };
	\node[state] (89)   at (4,0)     [label={below:{\footnotesize ${3332}$}}]{ };
	\node[state] (910)  at (3,0.5)   [label={below:{\footnotesize ${3322}$}}]{ };
	\node[state] (1011) at (2,0)     [label={below:{\footnotesize ${3222}$}}]{ };
	\node[state] (1112) at (0,1)     [label={left:{\footnotesize ${2221}$}}]{ };
	\node[state] (121)  at (1,1.5)   [label={left:{\footnotesize ${2211}$}}]{ };
	\node[state] (13)   at (2,2)     [label={left:{\footnotesize ${2114}$}},yshift=2mm,xshift=-2mm]{ };
	\node[state] (46)   at (4,2)     [label={right:{\footnotesize ${1443}$}},yshift=2mm,xshift=2mm]{ };
	\node[state] (79)   at (4,1)     [label={right:{\footnotesize ${4332}$}},yshift=-2mm,xshift=2mm]{ };
	\node[state] (1012) at (2,1)     [label={left:{\footnotesize ${3221}$}},yshift=-2mm,xshift=-2mm]{ };
	
	\path (12)   edge node { }(23)   
 	      (23)   edge node { }(34)
 	      (34)   edge node { }(45) edge node { }(46) 
 	      (45)   edge node { }(56)
 	      (56)   edge node { }(67)   
 	      (67)   edge node { }(78) edge node { }(79)
 	      (78)   edge node { }(89)   
 	      (89)   edge node { }(910)
 	      (910)  edge node { }(1011) edge node { }(1012)  
 	      (1011) edge node { }(1112)
 	      (1112) edge node { }(121)
 	      (121)  edge node { }(12) edge node { }(13)
 	      (13)   edge node { }(34)
 	      (46)   edge node { }(67)
 	      (79)   edge node { }(910)
 	      (1012) edge node { }(121);   
\end{tikzpicture}    
\vspace{-2mm}
\caption{ }
\end{subfigure}   
\caption{\small (a) A quasi-$(4,3)$-labeling for $^*C_{12}$ and (b) the corresponding $(4,4)$-labeling for $L\left(^*C_{12}\right)$.}
\label{fig: chords}
\end{figure}

\section{$\boldsymbol{\infty}$-digraph $\boldsymbol{C_n\cdot C_p}$}

\n An $\infty$\textit{-graph\textit{•}} is defined as two cycles joined at a vertex \cite{liu}. In this paper we use the notation $C_n \cdot C_p$ to represent an $\infty$\textit{-digraph} that is obtained by joining two dicycles $C_n$ and $C_p$ at a vertex. Since $L^{m+1}\left(C_n\cdot C_p\right)$ has more arcs and no less vertices than $L^{m}\left(C_n\cdot C_p\right)$ for any positive integers $n,p$ and $m$, Corollary \ref{cor: repeat} is usefully applicable to this graph family.

\vspace{3mm}
\begin{theorem}
\label{th: kupueven}
$C_n\cdot C_p$ has a quasi-$\left(4,\frac{n}{2}+1\right)$-labeling for any even $n\geq 4$, $n\leq p \leq \frac{5n}{2}+3$, and positive integer $m$.
\end{theorem}
\begin{proof}
Let the vertex set of $C_n\cdot C_p$ be $\{v_i\}_{i=1}^n \cup \{u_i\}_{i=1}^p$ where $v_2=u_2$ is the shared vertex.\\
Let the arc set be $\{v_iv_{i+1}\}_{i=1}^{n-1}\cup \{u_iu_{i+1}\}_{i=1}^{p-1} \cup \{v_nv_1,u_pu_1\}$.\\
\n Let the quasi-$\left(4,\frac{n}{2}+1\right)$-labeling for $C_n\cdot C_p$ defined as follows:\\[-7mm]
\begin{align*}
l(v_i)&=\begin{cases} 
(\underbrace{1,1,\ldots,1}_{\frac{n}{2}-i+2},\underbrace{2,2,\ldots,2}_{i-1}) & \text{if }1\leq i \leq \frac{n}{2} \\
(1,\underbrace{2,2,\ldots,2}_{\frac{n}{2}-1},1) & \text{if } i=\frac{n}{2}+1 \\ (\underbrace{2,2,\ldots,2}_{n-i+1},\underbrace{1,1,\ldots,1}_{i-\frac{n}{2}}) & \text{if }\frac{n}{2}+2\leq i \leq n \end{cases}
\end{align*}
\n For $l(u_i)$, first observe the case when $p=\frac{5n}{2}+3$:\\[-7mm]
\begin{align*}
l(u_i) &=\begin{cases}
(3,\underbrace{1,1,\ldots,1}_{\frac{n}{2}}) & \text{if } i=1\\
(\underbrace{1,1,\ldots,1}_{\frac{n}{2}-i+2},\underbrace{2,2,\ldots,2}_{i-1}) &\text{if } 2\leq i\leq \frac{n}{2}+1 \\
(\underbrace{2,2,\ldots,2}_{n-i+3},\underbrace{3,3,\ldots,3}_{i-\frac{n}{2}-2}) & \text{if } \frac{n}{2}+2\leq i \leq n+2\\
(\underbrace{3,3,\ldots,3}_{\frac{3n}{2}-i+3},\underbrace{4,4,\ldots,4}_{i-n-2}) &\text{if } n+3\leq i \leq \frac{3n}{2}+3 \\
(\underbrace{4,4,\ldots,4}_{2n-i+4},\underbrace{3,3,\ldots,3}_{i-\frac{3n}{2}-3}) &\text{if } \frac{3n}{2}+4\leq i \leq 2n+4\\
(\underbrace{3,3,\ldots,3}_{\frac{5n}{2}-i+5},\underbrace{1,1,\ldots,1}_{i-2n-4}) &\text{if } 2n+5\leq i \leq \frac{5n}{2}+3\\
\end{cases}
\end{align*}
\n This way we have $l(v_2)=(\underbrace{1,1,\ldots,1}_{\frac{n}{2}},2)=l(u_2)$ and all vertices have distinct label.\\
It is also immediately follows that $\alpha=4$ and $k=\frac{n}{2}+1$, and deBruijn property holds.\\
\n Adding more vertices to $C_p$ while preserving deBruijn property will force the existence of vertex label of the form $(\underbrace{2,2,\ldots,2}_{j},\underbrace{1,1,\ldots,1}_{k-j})$ for some $1\leq j \leq k$. However, this form of vertex label is already used to label some vertices in $C_n$, so the distinctive property of the quasi labeling will be violated. Hence no more vertex can be added to $C_p$ and we have $p\leq \frac{5n}{2}+3$.\\
\n Next observe that some vertices can be omitted by merging some vertices. For example, the vertex with sequence of labels 122, 222, 223 can be merged into 122, 223, and further to 123. Hence we can omits vertex one by one until we reach the minimum number of vertices in $C_p$ that is determined by the fixed $l(u_2)$. Since $l(u_2)=(\underbrace{1,1,\ldots,1}_{\frac{n}{2}},2)$, the shortest possible sequence of vertex labels is  
$$(3,\underbrace{1,1,\ldots,1}_{k-1}),(\underbrace{1,1,\ldots,1}_{k-1},2),(\underbrace{1,1,\ldots,1}_{k-2},2,3),(\underbrace{1,1,\ldots,1}_{k-3},2,3,1),$$
$$(\underbrace{1,1,\ldots,1}_{k-4},2,3,1,1),\ldots,(1,2,3,\underbrace{1,1,\ldots,1}_{k-3}), (2,3,\underbrace{1,1,\ldots,1}_{k-2})$$

\n The number ``3" in the labels above can be replaced with ``4", but the use of ``1" and ``2" as well as the existence of ``3" (or ``4") are forced by $l(u_2)$ and the distinctive property of the quasi labeling.\\
\n There are $k+1=\ceil*{\frac{n}{2}}+2$ labels (vertices) in above sequence, and since $\ceil*{\frac{n}{2}}+2\leq n$ for all $n \geq 4$, we have that quasi-$\left(4,\frac{n}{2}+1\right)$-labeling exists for all $n\leq p \leq \frac{5n}{2}+3$.
\end{proof}

\n In Figure \ref{fig: butteven} the dicycle is short enough so $\alpha=3$ is sufficient to label all vertices. Labeling in (a) obtained from labeling in (b) by merging the vertices labeled 2333, 3333, 3331 into 2333, 3331.

\begin{figure}[h]
\begin{subfigure}{.5\textwidth}
\centering
\begin{tikzpicture}[->,>=stealth',shorten >=1pt,auto, semithick,scale=1.2]
\tikzstyle{every state}=[draw,circle,thick, fill=black, minimum size=5pt,
                            inner sep=0pt]
	\node[state] (1112)  at (0:1)   
	             [label={right:{\footnotesize ${1112}$}}]{ };
	\node[state] (1111)  at (60:1)  
	             [label={above:{\footnotesize ${1111}$}}]{ };
	\node[state] (2111)  at (120:1) 
	             [label={above:{\footnotesize ${2111}$}}]{ };
	\node[state] (2211)  at (180:1) 
	             [label={left:{\footnotesize ${2211}$}}]{ };
	\node[state] (1221)  at (240:1) 
	             [label={below:{\footnotesize ${1221}$}}]{ };
	\node[state] (1122)  at (300:1) 
	             [label={below:{\footnotesize ${1122}$}}]{ };
	\node[state,xshift=2.8cm] (3331)  at (180-3*51.4:1.2) 
	                          [label={right:{\footnotesize ${3331}$}}]{ };
	\node[state,xshift=2.8cm] (3311)  at (180-2*51.4:1.2) 
	                          [label={above:{\footnotesize ${3311}$}}]{ };
	\node[state,xshift=2.8cm] (3111)  at (180-51.4:1.2) 
	                          [label={above:{\footnotesize ${3111}$}}]{ };
	\node[state,xshift=2.8cm] (1123)  at (180+51.4:1.2) 
	                          [label={below:{\footnotesize ${1123}$}}]{ };
	\node[state,xshift=2.8cm] (1233)  at (180+2*51.4:1.2) 
	                          [label={below:{\footnotesize ${1233}$}}]{ };
	\node[state,xshift=2.8cm] (2333)  at (180+3*51.4:1.2) 
	                          [label={right:{\footnotesize ${2333}$}}]{ };
	
	\path (1112) edge node { } (1122) edge node { } (1123)
	      (1122) edge node { } (1221)
	      (1221) edge node { } (2211)
	      (2211) edge node { } (2111)
	      (2111) edge node { } (1111)
	      (1111) edge node { } (1112)
	      (1123) edge node { } (1233)
	      (1233) edge node { } (2333)
	      (2333) edge node { } (3331)
	      (3331) edge node { } (3311)
	      (3311) edge node { } (3111)
	      (3111) edge node { } (1112);
\end{tikzpicture}    
\vspace{-2mm}
\caption{ }
\end{subfigure} 
\begin{subfigure}{.5\textwidth}
\centering
\begin{tikzpicture}[->,>=stealth',shorten >=1pt,auto, semithick,scale=1.2]
\tikzstyle{every state}=[draw,circle,thick, fill=black, minimum size=5pt,
                            inner sep=0pt]
	\node[state] (1112)  at (0:1)   
	             [label={right:{\footnotesize ${1112}$}}]{ };
	\node[state] (1111)  at (60:1)  
	             [label={above:{\footnotesize ${1111}$}}]{ };
	\node[state] (2111)  at (120:1) 
	             [label={above:{\footnotesize ${2111}$}}]{ };
	\node[state] (2211)  at (180:1) 
	             [label={left:{\footnotesize ${2211}$}}]{ };
	\node[state] (1221)  at (240:1) 
	             [label={below:{\footnotesize ${1221}$}}]{ };
	\node[state] (1122)  at (300:1) 
	             [label={below:{\footnotesize ${1122}$}}]{ };
	\node[state,xshift=2.8cm] (3331)  at (180-3*45:1.2) 
	                          [label={right:{\footnotesize ${3331}$}}]{ };
	\node[state,xshift=2.8cm] (3311)  at (180-2*45:1.2) 
	                          [label={above:{\footnotesize ${3311}$}}]{ };
	\node[state,xshift=2.8cm] (3111)  at (180-45:1.2) 
	                          [label={above:{\footnotesize ${3111}$}}]{ };
	\node[state,xshift=2.8cm] (1123)  at (180+45:1.2) 
	                          [label={below:{\footnotesize ${1123}$}}]{ };
	\node[state,xshift=2.8cm] (1233)  at (180+2*45:1.2) 
	                          [label={below:{\footnotesize ${1233}$}}]{ };
	\node[state,xshift=2.8cm] (2333)  at (180+3*45:1.2) 
	                          [label={right:{\footnotesize ${2333}$}}]{ };
	\node[state,xshift=2.8cm] (3333)  at (180+4*45:1.2) 
	                          [label={right:{\footnotesize ${3333}$}}]{ };
	                          
	\path (1112) edge node { } (1122) edge node { } (1123)
	      (1122) edge node { } (1221)
	      (1221) edge node { } (2211)
	      (2211) edge node { } (2111)
	      (2111) edge node { } (1111)
	      (1111) edge node { } (1112)
	      (1123) edge node { } (1233)
	      (1233) edge node { } (2333)
	      (2333) edge node { } (3333)
	      (3333) edge node { } (3331)
	      (3331) edge node { } (3311)
	      (3311) edge node { } (3111)
	      (3111) edge node { } (1112);
\end{tikzpicture}     
\vspace{-2mm}
\caption{ }
\end{subfigure}   
\vspace{-2mm}
\caption{\small (a) A Quasi-$(4,4)$-labeling of $C_6\cdot C_7$ and (b) A Quasi-$(4,4)$-labeling of $C_6\cdot C_8$.}
\label{fig: butteven}
\end{figure}
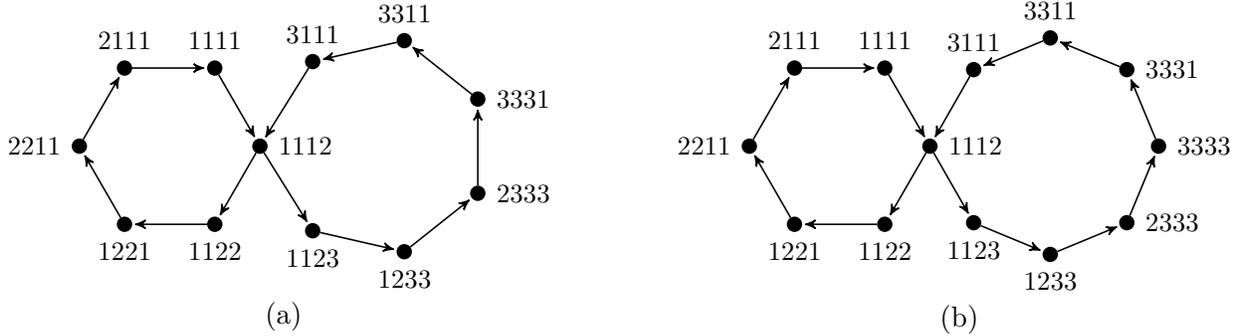

\vspace{3mm}
\begin{theorem}
\label{th: kupuodd}
$C_n\cdot C_p$ has a quasi-$\left(4,\ceil*{\frac{n}{2}}+1\right)$-labeling for any odd $n> 4$, $n\leq p \leq 5\ceil*{\frac{n}{2}}+3$, and positive integer $m$.
\end{theorem}
\begin{proof}
Let the vertex set of $C_n\cdot C_p$ be $\{v_i\}_{i=1}^n \cup \{u_i\}_{i=1}^p$ where $v_2=u_2$ is the shared vertex.\\
Let the arc set be $\{v_iv_{i+1}\}_{i=1}^{n-1}\cup \{u_iu_{i+1}\}_{i=1}^{p-1} \cup \{v_nv_1,u_pu_1\}$.\\
\n Let the quasi-$(4,\ceil*{\frac{n}{2}}+1)$-labeling $l$ for $C_n\cdot C_p$ defined as follows:\\[-7mm]
\begin{align*}
l(v_i)&=\begin{cases} 
(\underbrace{1,1,\ldots,1}_{\ceil*{\frac{n}{2}}-i+2},\underbrace{2,2,\ldots,2}_{i-1}) & \text{if }1\leq i \leq \ceil*{\frac{n}{2}}-1 \\
(1,1,\underbrace{2,2,\ldots,2}_{\ceil*{\frac{n}{2}}-2},1) & \text{if }i=\ceil*{\frac{n}{2}}   \\
(1,\underbrace{2,2,\ldots,2}_{\ceil*{\frac{n}{2}}-2},1,1) & \text{if }i=\ceil*{\frac{n}{2}}+1   \\
(\underbrace{2,2,\ldots,2}_{n-i+1},\underbrace{1,1,\ldots,1}_{i-\ceil*{\frac{n}{2}}+1}) & \text{if }\ceil*{\frac{n}{2}}+2\leq i \leq n \end{cases}
\end{align*}
\n For $l(u_i)$, first observe the case when $p=5\ceil*{\frac{n}{2}}+3$:\\[-7mm]
\begin{align*}
l(u_i)&=\begin{cases}
(3,\underbrace{1,1,\ldots,1}_{\ceil*{\frac{n}{2}}}) & \text{if } i=1\\
(\underbrace{1,1,\ldots,1}_{\ceil*{\frac{n}{2}}-i+2},\underbrace{2,2,\ldots,2}_{i-1}) &\text{if } 2\leq i\leq \ceil*{\frac{n}{2}}+1 \\
(\underbrace{2,2,\ldots,2}_{n-i+4},\underbrace{3,3,\ldots,3}_{i-\ceil*{\frac{n}{2}}-2}) & \text{if } \ceil*{\frac{n}{2}}+2\leq i \leq n+3\\
(\underbrace{3,3,\ldots,3}_{3\ceil*{\frac{n}{2}}-i+3},\underbrace{4,4,\ldots,4}_{i-n-3}) &\text{if } n+4\leq i \leq 3\ceil*{\frac{n}{2}}+3 \\
(\underbrace{4,4,\ldots,4}_{2n-i+6},\underbrace{3,3,\ldots,3}_{i-3\ceil*{\frac{n}{2}}-3}) &\text{if } 3\ceil*{\frac{n}{2}}+4\leq i \leq 2n+6\\
(\underbrace{3,3,\ldots,3}_{5\ceil*{\frac{n}{2}}-i+5},\underbrace{1,1,\ldots,1}_{i-2n-6}) &\text{if } 2n+7\leq i \leq 5\ceil*{\frac{n}{2}}+3\\
\end{cases}
\end{align*}
\n This way we have $l(v_2)=(\underbrace{1,1,\ldots,1}_{\ceil*{\frac{n}{2}}},2)=l(u_2)$ and all vertices have distinct label.\\
It is also immediately follows that $\alpha=4$ and $k=\ceil*{\frac{n}{2}}+1$, and deBruijn property holds.\\
\n Using similar arguments as in Theorem \ref{th: kupueven}, we have the labeling exists for $n\leq p \leq \frac{5n}{2}+3$.
\end{proof}

\vspace{3mm}
\begin{cor}
$L^m(C_n \cdot C_p)$ are DNA graphs for $n\geq 4$, $n\leq p \leq 5\ceil*{\frac{n}{2}}+3$, and $m\in \mathbb{Z}^+$.
\end{cor}
\begin{proof}
Apply Corollary \ref{cor: repeat} to Theorem \ref{th: kupueven} and Theorem \ref{th: kupuodd}.
\end{proof}

\vspace{-2mm}
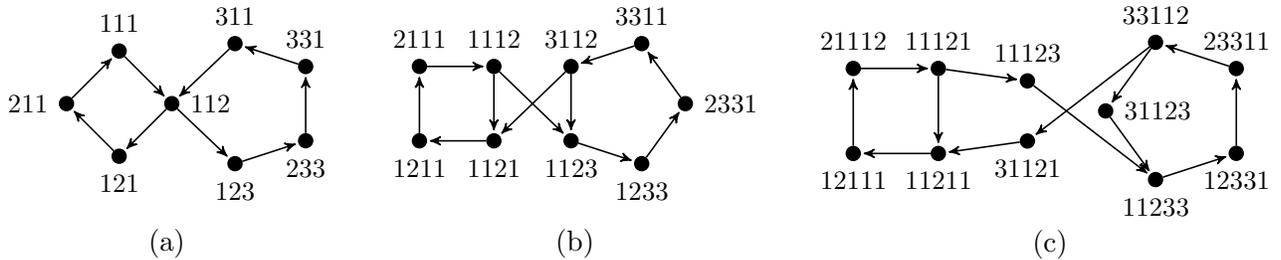
\begin{figure}[h]
\captionsetup{justification=centering}
\begin{subfigure}{.3\textwidth}
\centering
\begin{tikzpicture}[->,>=stealth',shorten >=1pt,auto, semithick,scale=0.7]
\tikzstyle{every state}=[draw,circle,thick, fill=black, minimum size=5pt,
                            inner sep=0pt]
	\node[state] (112)  at (0:1)   
	             [label={right:{\footnotesize ${112}$}}]{ };
	\node[state] (111)  at (90:1)  
	             [label={above:{\footnotesize ${111}$}}]{ };
	\node[state] (211)  at (180:1) 
	             [label={left:{\footnotesize ${211}$}}]{ };
	\node[state] (121)  at (270:1) 
	             [label={below:{\footnotesize ${121}$}}]{ };
	\node[state,xshift=1.8cm] (331)  at (180-2*72:1.2) 
	                          [label={above:{\footnotesize ${331}$}}]{ };
	\node[state,xshift=1.8cm] (311)  at (180-72:1.2) 
	                          [label={above:{\footnotesize ${311}$}}]{ };
	\node[state,xshift=1.8cm] (123)  at (180+72:1.2) 
	                          [label={below:{\footnotesize ${123}$}}]{ };
	\node[state,xshift=1.8cm] (233)  at (180+2*72:1.2) 
	                          [label={below:{\footnotesize ${233}$}}]{ };
	
	\path (112) edge node { } (121) edge node { } (123)
	      (121) edge node { } (211)
	      (211) edge node { } (111)
	      (111) edge node { } (112)
	      (123) edge node { } (233)
	      (233) edge node { } (331)
	      (331) edge node { } (311)
	      (311) edge node { } (112);
\end{tikzpicture} 
\caption{ }
\end{subfigure}   
\begin{subfigure}{.3\textwidth}
\centering
\begin{tikzpicture}[->,>=stealth',shorten >=1pt,auto, semithick,scale=0.7]
\tikzstyle{every state}=[draw,circle,thick, fill=black, minimum size=5pt,
                            inner sep=0pt]
	\node[state] (1112)  at (45:1)   
	             [label={above:{\footnotesize ${1112}$}}]{ };
	\node[state] (2111)  at (135:1)  
	             [label={above:{\footnotesize ${2111}$}}]{ };
	\node[state] (1211)  at (225:1) 
	             [label={below:{\footnotesize ${1211}$}}]{ };
	\node[state] (1121)  at (315:1) 
	             [label={below:{\footnotesize ${1121}$}}]{ };
	\node[state,xshift=2.2cm] (2331)  at (0:1.2) 
	                          [label={right:{\footnotesize ${2331}$}}]{ };
	\node[state,xshift=2.2cm] (3311)  at (72:1.2) 
	                          [label={above:{\footnotesize ${3311}$}}]{ };
	\node[state,xshift=2.2cm] (3112)  at (2*72:1.2) 
	                          [label={above:{\footnotesize ${3112}$}}]{ };
	\node[state,xshift=2.2cm] (1123)  at (3*72:1.2) 
	                          [label={below:{\footnotesize ${1123}$}}]{ };
	\node[state,xshift=2.2cm] (1233)  at (4*72:1.2) 
	                          [label={below:{\footnotesize ${1233}$}}]{ };	
	
	\path (1112) edge node { } (1121) edge node { } (1123)
	      (1121) edge node { } (1211)
	      (1211) edge node { } (2111)
	      (2111) edge node { } (1112)
	      (1123) edge node { } (1233)
	      (1233) edge node { } (2331)
	      (2331) edge node { } (3311)
	      (3311) edge node { } (3112)
	      (3112) edge node { } (1123) edge node { } (1121)  ;
\end{tikzpicture} 
\caption{ }
\end{subfigure}   
\begin{subfigure}{.4\textwidth}
\centering
\begin{tikzpicture}[->,>=stealth',shorten >=1pt,auto, semithick,scale=0.8]
\tikzstyle{every state}=[draw,circle,thick, fill=black, minimum size=5pt,
                            inner sep=0pt]
	\node[state] (21112)  at (135:1)   
	             [label={above:{\footnotesize ${21112}$}}]{ };
	\node[state] (12111)  at (225:1)  
	             [label={below:{\footnotesize ${12111}$}}]{ };
	\node[state] (11211)  at (315:1) 
	             [label={below:{\footnotesize ${11211}$}}]{ };
	\node[state] (11121)  at (45:1) 
	             [label={above:{\footnotesize ${11121}$}}]{ };
	
	\node[state,xshift=1.75cm] (11123)  at (90:0.5) 
	                          [label={above:{\footnotesize ${11123}$}}]{ };
	\node[state,xshift=1.75cm] (31121)  at (270:0.5) 
	                          [label={below:{\footnotesize ${31121}$}}]{ };
	                          
	\node[state,xshift=3.75cm] (23311)  at (180-2*72:1.2) 
	                          [label={above:{\footnotesize ${23311}$}}]{ };
	\node[state,xshift=3.75cm] (33112)  at (180-72:1.2) 
	                          [label={above:{\footnotesize ${33112}$}}]{ };
	\node[state,xshift=3.75cm] (31123)  at (180:1.2) 
	                          [label={right:{\footnotesize ${31123}$}}]{ };                          
	\node[state,xshift=3.75cm] (11233)  at (180+72:1.2) 
	                          [label={below:{\footnotesize ${11233}$}}]{ };
	\node[state,xshift=3.75cm] (12331)  at (180+2*72:1.2) 
	                          [label={below:{\footnotesize ${12331}$}}]{ };
	
	\path (11121) edge node { } (11211) edge node { } (11123)
	      (11211) edge node { } (12111)
	      (12111) edge node { } (21112)
	      (21112) edge node { } (11121)
		  (11123) edge node { } (11233)
		  (11233) edge node { } (12331)
		  (12331) edge node { } (23311)
		  (23311) edge node { } (33112)
		  (33112) edge node { } (31123) edge node { } (31121)
		  (31123) edge node { } (11233)	      
	      (31121) edge node { } (11211);
\end{tikzpicture} 
\vspace{-2mm}
\caption{ }
\end{subfigure}   
\vspace{-2mm}
\caption{\small (a) A Quasi-$(4,3)$-labeling of $C_4\cdot C_5$ and the corresponding (b) $(4,4)$-labeling of $L(C_4\cdot C_5)$ and (c) $(4,5)$-labeling of $L^2(C_4\cdot C_5)$. Digraphs $L(C_4\cdot C_5)$ and $L^2(C_4\cdot C_5)$ are DNA graphs.}
\label{fig: DNAbutt}
\end{figure}

\vspace{-1mm}
\begin{theorem}
$L^m(C_3 \cdot C_p)$ are DNA graphs for $4\leq p \leq 13$ and $m\in \mathbb{Z}^+$.
\label{th: kupu3}
\end{theorem}
\begin{proof}
Use the construction in the proof of Theorem \ref{th: kupueven} for $n=4$, but change the labeling for $C_n$ into $l(v_1),l(v_2),l(v_3)=211,112,121$.
\end{proof}

\vfill
\n Since DNA spectrum obtained from an experiment can have random length, the spectrum are normally chopped into relatively small pieces before the sequencing process. Therefore it is preferable to have a labeling construction with relatively shorter length. Theorem \ref{th: kupukupu} shows that a quasi-$(4,k)$-labeling with the shorter length of $k=\ceil*{\frac{n}{2}}$ is possible for $C_n\cdot C_n$.\\[-2mm]

\pagebreak
\begin{theorem}
\label{th: kupukupu}
$L^m\left(C_n\cdot C_n\right)$ are DNA graph for any $n\geq 3$ and positive integer $m$. 
\end{theorem}
\begin{proof}
Let the vertex set of the dicycles be $\{v_i\}_{i=1}^n \cup \{u_i\}_{i=1}^n$ where $v_2=u_2$ is the shared vertex. Let the arc set be $\{v_iv_{i+1},u_iu_{i+1}\}_{i=1}^{n-1}\cup \{v_nv_1,u_nu_1\}$.\\
For $n=3$, let $l(v_1),l(v_2),l(v_3)=11,12,21$ and $l(u_1),l(u_2),l(u_3)=31,12,23$.\\
For $n\geq 4$, let the quasi-$\left(3,\ceil*{\frac{n}{2}}\right)$-labeling $l$ for $C_n\cdot C_n$ defined as follows:\\[-7mm]
\begin{align*}
    l(v_i)&=\begin{cases} (\underbrace{1,1,\ldots,1}_{\ceil{\frac{n}{2}}-i+1},\underbrace{2,2,\ldots,2}_{i-1}) & \text{if }1\leq i \leq \ceil*{\frac{n}{2}} \\ (\underbrace{2,2,\ldots,2}_{n-i+1},\underbrace{1,1,\ldots,1}_{i-\ceil{\frac{n}{2}}}) & \text{if }\ceil{\frac{n}{2}}+1\leq i \leq n  \text{ and }n \text{ is odd}\\
(\underbrace{2,2,\ldots,2}_{n-i+1},\underbrace{1,1,\ldots,1}_{i-\frac{n}{2}-1}) & \text{if }\frac{n}{2}+1\leq i \leq n  \text{ and }n \text{ is even}
    \end{cases}\\
    l(u_i)&=\begin{cases} (3,\underbrace{1,1,\ldots,1}_{\ceil*{\frac{n}{2}}-1}) & \text{if }i=1 \\ (\underbrace{1,1,\ldots,1}_{\ceil{\frac{n}{2}}-i+1},2,\underbrace{3,3,\ldots,3}_{i-2}) & \text{if } 2\leq i \leq \ceil*{\frac{n}{2}}+1 \\ (\underbrace{3,3,\ldots,3}_{n-i+2},\underbrace{1,1,\ldots,1}_{i-\ceil{\frac{n}{2}}-1}) & \text{if } \ceil*{\frac{n}{2}}+2 \leq i \leq n \text{ and }n \text{ is odd}\\
 (\underbrace{3,3,\ldots,3}_{n-i+2},\underbrace{1,1,\ldots,1}_{i-\frac{n}{2}-2}) & \text{if } \frac{n}{2}+2 \leq i \leq n \text{ and }n \text{ is even}   
\end{cases}
\end{align*}
This way we have $l(v_2)=(\underbrace{1,1,\ldots,1}_{\ceil*{\frac{n}{2}}-1},2)=l(u_2)$ and all vertices have distinct label.\\
It is also immediately follows that $\alpha=3$ and $k=\ceil*{\frac{n}{2}}$, and deBruijn property holds.
\end{proof}

\vspace{3mm}
\n Observe that $L(C_4\cdot C_4)\cong P_2 \Box P_4$, which is a ladder digraph as shown in Figure \ref{fig: ladder}.\\[-6mm] 

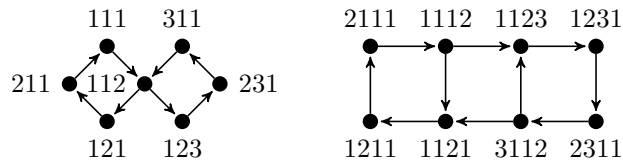
\begin{figure}[h]
\centering
\begin{tikzpicture}[->,>=stealth',shorten >=1pt,auto, semithick, scale=0.5]
\tikzstyle{every state}=[draw,circle,thick, fill=black, minimum size=5pt,
                            inner sep=0pt]
	\node[state] (0)  at (0,1) [label={left:{\footnotesize ${211}$}}]{ };
	\node[state] (10) at (1,0) [label={below:{\footnotesize ${121}$}}]{ };
	\node[state] (11) at (1,2) [label={above:{\footnotesize ${111}$}}]{ };
	\node[state] (2)  at (2,1) [label={left:{\footnotesize ${112}$}}]{ };
	\node[state] (30) at (3,0) [label={below:{\footnotesize ${123}$}}]{ };
	\node[state] (31) at (3,2) [label={above:{\footnotesize ${311}$}}]{ };
	\node[state] (4)  at (4,1) [label={right:{\footnotesize ${231}$}}]{ };
	
	\path (0)  edge node { }(11)   
 	      (10) edge node { }(0)
 	      (11) edge node { }(2)
 	      (2)  edge node { }(30) edge node { }(10)  
 	      (30) edge node { }(4)
 	      (4)  edge node { }(31)
 	      (31) edge node { }(2)	;
	
	\node[state] (50) at (8,0) [label={below:{\footnotesize ${1211}$}}]{ };
	\node[state] (51) at (8,2) [label={above:{\footnotesize ${2111}$}}]{ };
	\node[state] (60) at (10,0) [label={below:{\footnotesize ${1121}$}}]{ };
	\node[state] (61) at (10,2) [label={above:{\footnotesize ${1112}$}}]{ };
	\node[state] (70) at (12,0) [label={below:{\footnotesize ${3112}$}}]{ };
	\node[state] (71) at (12,2) [label={above:{\footnotesize ${1123}$}}]{ };
	\node[state] (80) at (14,0) [label={below:{\footnotesize ${2311}$}}]{ };
	\node[state] (81) at (14,2) [label={above:{\footnotesize ${1231}$}}]{ };
	
	\path (50) edge node { }(51)
	      (51) edge node { }(61)
	      (61) edge node { }(71) edge node { }(60)
	      (71) edge node { }(81)
	      (81) edge node { }(80)
	      (80) edge node { }(70)
	      (70) edge node { }(60) edge node { }(71)
	      (60) edge node { }(50);
\end{tikzpicture}    
\vspace{-2mm}
\caption{\small A Quasi-$(3,3)$-labeling of $C_4\cdot C_4$ and the corresponding $(3,4)$-labeling of $L(C_4\cdot C_4)\cong P_2 \Box P_4$.}
\label{fig: ladder}
\end{figure}

\n Taking $P_2\Box P_3$ as an induced subgraph of $P_2\Box P_4$ and preserving the labeling, we have ladder digraph $P_2\Box P_3$ is a DNA graph, with one of the corresponding $(3,4)$-labeling given in Figure \ref{fig: laddercil} (left), while the $(3,4)$-labeling of $P_2\Box P_5$ is given in Figure \ref{fig: laddercil} (right).\\[-6mm]

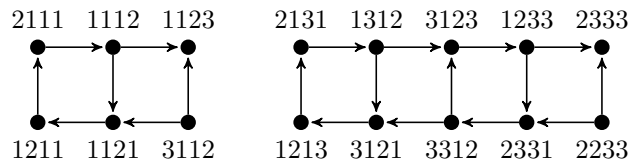
\begin{figure}[h]
\centering
\begin{tikzpicture}[->,>=stealth',shorten >=1pt,auto, semithick, scale=0.5]
\tikzstyle{every state}=[draw,circle,thick, fill=black, minimum size=5pt,
                            inner sep=0pt]
	\node[state] (00) at (0,0) [label={below:{\footnotesize ${1211}$}}]{ };
	\node[state] (01) at (0,2) [label={above:{\footnotesize ${2111}$}}]{ };
	\node[state] (10) at (2,0) [label={below:{\footnotesize ${1121}$}}]{ };
	\node[state] (11) at (2,2) [label={above:{\footnotesize ${1112}$}}]{ };
	\node[state] (20) at (4,0) [label={below:{\footnotesize ${3112}$}}]{ };
	\node[state] (21) at (4,2) [label={above:{\footnotesize ${1123}$}}]{ };

	\path (00) edge node { }(01)
	      (01) edge node { }(11)
	      (11) edge node { }(21) edge node { }(10)
	      (20) edge node { }(10) edge node { }(21)
	      (10) edge node { }(00);
	      
	\node[state] (40) at (7,0)  [label={below:{\footnotesize ${1213}$}}]{ };
	\node[state] (41) at (7,2)  [label={above:{\footnotesize ${2131}$}}]{ };	
	\node[state] (50) at (9,0)  [label={below:{\footnotesize ${3121}$}}]{ };
	\node[state] (51) at (9,2)  [label={above:{\footnotesize ${1312}$}}]{ };
	\node[state] (60) at (11,0) [label={below:{\footnotesize ${3312}$}}]{ };
	\node[state] (61) at (11,2) [label={above:{\footnotesize ${3123}$}}]{ };
	\node[state] (70) at (13,0) [label={below:{\footnotesize ${2331}$}}]{ };
	\node[state] (71) at (13,2) [label={above:{\footnotesize ${1233}$}}]{ };
	\node[state] (80) at (15,0) [label={below:{\footnotesize ${2233}$}}]{ };
	\node[state] (81) at (15,2) [label={above:{\footnotesize ${2333}$}}]{ };
	
	\path (40) edge node { }(41)
		  (41) edge node { }(51)
		  (51) edge node { }(61) edge node { }(50)
	      (61) edge node { }(71) 
	      (71) edge node { }(81) edge node { }(70)
	      (80) edge node { }(70) edge node { }(81)
	      (70) edge node { }(60) 
	      (60) edge node { }(50) edge node { }(61)     
	      (50) edge node { }(40);
      
\end{tikzpicture} 
\vspace{-2mm}   
\caption{\small A $(3,4)$-labeling of $P_2 \Box P_3$ and a $(3,4)$-labeling of $P_2\Box P_5$.}
\label{fig: laddercil}
\end{figure}

\vspace{-2mm}
\begin{con}
For all integers $n\geq 3$, ladder digraphs $P_2\Box P_n$ are DNA graphs.
\end{con}

\section{3-blade-propeller $\boldsymbol{C_n\cdot C_p \cdot C_q}$}

\n Define a \textit{3-blade-propeller}, denoted by $C_n\cdot C_p \cdot C_q$, as a digraph obtained by joining three dicycles $C_n, C_p$, and $C_q$ at a vertex. The name is to avoid any confusion with a \textit{propeller graph} that defined in \cite{liu}. When $p=q=n$, a 3-blade-propeller is the \textit{windmill} graph $D_n^3$. Since $L^{m+1}\left(C_n\cdot C_p \cdot C_q\right)$ has more arcs and no less vertices than $L^{m}\left(C_n\cdot C_p \cdot C_q\right)$ for any positive integers $n,p$ and $m$, Corollary \ref{cor: repeat} is usefully applicable to this graph family.\\[-2mm]

\begin{theorem}
$L^m\left(D_n^3\right)$ are DNA graph for $n\geq 3$ and positive integer $m$.
\label{th: windmill}
\end{theorem}
\begin{proof}
Let the vertex set of the dicycles are $\{v_i^j\}_{i=1}^n$ where $v_i^j$ denotes the $i\textsuperscript{th}$ vertex of the $j\textsuperscript{th}$ dicycle. Let $v_2^1=v_2^2=v_2^3$ be the shared vertex. Let the arc set be $\{v_i^jv_{i+1}^j\}_{i=1}^{n-1}\cup \{v_n^jv_1^j\}$ for $j=1,2,3$.\\
\n Define the quasi-$\left(4,\ceil*{\frac{n}{2}}\right)$-labeling $l$ as follows:\\[-7mm]
\begin{align*}
l(v_i^j) &=\begin{cases}
(\underbrace{1,1,\ldots,1}_{\ceil*{\frac{n}{2}}}) & \text{if }i=1, j=1\\
((j+1),\underbrace{1,1,\ldots,1}_{\ceil*{\frac{n}{2}}-1}) & \text{if }i=1, j=2,3\\
(\underbrace{1,1,\ldots,1}_{\ceil*{\frac{n}{2}}-1},2) & \text{if }i=2\\
(\underbrace{1,1,\ldots,1}_{\ceil*{\frac{n}{2}}-i+1},\underbrace{2,2,\ldots,2}_{i-1}) & \text{if }3\leq i\leq \ceil*{\frac{n}{2}}, j=1\\
(\underbrace{1,1,\ldots,1}_{\ceil*{\frac{n}{2}}-i+1},2,\underbrace{(j+1),(j+1),\ldots,(j+1)}_{i-2}) & \text{if }3\leq i\leq \ceil*{\frac{n}{2}}+1, j=2,3\\
(\underbrace{2,2,\ldots,2}_{n-i+1},\underbrace{1,1,\ldots,1}_{i-\ceil*{\frac{n}{2}}}) & \text{if }\ceil*{\frac{n}{2}}+1 \leq i\leq n, j=1 \text{ and }n \text{ is odd}\\
(\underbrace{2,2,\ldots,2}_{n-i+1},\underbrace{1,1,\ldots,1}_{i-\frac{n}{2}-1}) & \text{if }\frac{n}{2}+1 \leq i\leq n, j=1 \text{ and }n \text{ is even}\\
(\underbrace{(j+1),(j+1),\ldots,(j+1)}_{n-i+2},\underbrace{1,1,\ldots,1}_{i-\ceil*{\frac{n}{2}}-1}) & \text{if }\ceil*{\frac{n}{2}}+2 \leq i \leq n, j=2,3\text{ and }n \text{ is odd}\\
(\underbrace{(j+1),(j+1),\ldots,(j+1)}_{n-i+2},\underbrace{1,1,\ldots,1}_{i-\frac{n}{2}-2}) & \text{if }\frac{n}{2}+2 \leq i \leq n, j=2,3\text{ and }n \text{ is even}\\
\end{cases}
\end{align*}
\n It is clear from the definition that $l(v_2^1)=l(v_2^2)=l(v_2^3)$ and all vertices have distinct label. It is also immediately follows that $\alpha=4$, $k=\ceil*{\frac{n}{2}}$, and deBruijn property holds. The theorem then follows from Corollary \ref{cor: repeat}.
\end{proof}

\vspace{3mm}
\begin{theorem}
$L^m\left(C_n\cdot C_p\cdot C_q\right)$ are DNA graph for $n\geq 4$, $p,q\in \{n,n+1,n+2\}$, and $m\in \mathbb{Z}^+$.
\label{th: propeller}
\end{theorem}
\begin{proof}
Let the vertex set of the dicycles are $\{v_i^j\}_{i=1}^n$ where $v_i^j$ denotes the $i\textsuperscript{th}$ vertex of the $j\textsuperscript{th}$ dicycle. Let $v_2^1=v_2^2=v_2^3$ be the shared vertex. Let the arc set be $\{v_i^jv_{i+1}^j\}_{i=1}^{n-1}\cup \{v_n^jv_1^j\}$ for $j=1,2,3$.\\
\n First define the quasi-$\left(4,\frac{n}{2}+1\right)$-labeling $l^*$ for even values of $n$:\\[-7mm]
\begin{align*}
l^*(v_i^j)&=\begin{cases}
((j+1),\underbrace{1,1,\ldots,1}_{\frac{n}{2}}) & \text{if } i=1\\
(\underbrace{1,1,\ldots,1}_{\frac{n}{2}},2) & \text{if } i=2\\
(\underbrace{1,1,\ldots,1}_{\frac{n}{2}-i+2},2,\underbrace{(j+1),(j+1),\ldots,(j+1)}_{i-2}) & \text{if }3 \leq i \leq \frac{n}{2}+1\\
(2,\underbrace{(j+1),(j+1),\ldots,(j+1)}_{\frac{n}{2}-1},1) & \text{if }i=\frac{n}{2}+2\\
(\underbrace{(j+1),(j+1),\ldots,(j+1)}_{n-i+2},\underbrace{1,1,\ldots,1}_{i-\frac{n}{2}-1}) & \text{if }\frac{n}{2}+3\leq i\leq n
\end{cases}
\end{align*}
\n For odd $n$, use $l^*(v_{n+1}^j)$ and merge the two vertices labeled
 $(1,2,\underbrace{(j+1),\ldots,(j+1)}_{\frac{n}{2}-1})$ and\\ $(2,\underbrace{(j+1),\ldots,(j+1)}_{\frac{n}{2}-1},1)$ \n into a single vertex labeled $(1,2,\underbrace{(j+1),\ldots,(j+1)}_{\frac{n}{2}-2},1)$.\\
\n It is clear from $l^*$ definition that $l^*(v_2^1)=l^*(v_2^2)=l^*(v_2^3)$ and all vertices have distinct label. It is also immediately follows that $\alpha=4$, $k=\ceil*{\frac{n}{2}}+1$, and deBruijn property holds.
\n Combining $l^*$ and $l$ that given in the proof of Theorem \ref{th: windmill}, we can have either quasi-$\left(4,\ceil*{\frac{n}{2}}\right)$-labeling or quasi-$\left(4,\ceil*{\frac{n}{2}}+1\right)$-labeling of $C_n\cdot C_p\cdot C_q$. The theorem then follows from Corollary \ref{cor: repeat}.
\end{proof}

\begin{figure}[h]
\begin{subfigure}{.5\textwidth}
\centering
\begin{tikzpicture}[->,>=stealth',shorten >=1pt,auto, semithick,scale=0.7]
\tikzstyle{every state}=[draw,circle,thick, fill=black, minimum size=5pt,
                            inner sep=0pt]
	\node[state] (221)  at (0,1)   
	             [label={left:{\footnotesize ${221}$}}]{ };
	\node[state] (211)  at (1,0)  
	             [label={below:{\footnotesize ${211}$}}]{ };
	\node[state] (111)  at (2,1) 
	             [label={below:{\footnotesize ${\text{ }\;\;111}$}}]{ };
	\node[state] (112)  at (3,3) 
	             [label={right:{\footnotesize ${112}$}},yshift=2mm]{ };
	\node[state] (122)  at (1,2) 
	             [label={above:{\footnotesize ${122\;\;\text{ }}$}}]{ };
	
	\node[state] (123)  at (4,1)   
	             [label={below:{\footnotesize ${123\;\;\text{ }}$}}]{ };
	\node[state] (233)  at (5,0)  
	             [label={below:{\footnotesize ${233}$}}]{ };
	\node[state] (331)  at (6,1) 
	             [label={right:{\footnotesize ${331}$}}]{ };
	\node[state] (311)  at (5,2) 
	             [label={above:{\footnotesize ${\text{ }\;\;311}$}}]{ };
	             
	\node[state] (124)  at (4,5)   
	             [label={right:{\footnotesize ${124}$}}]{ };
	\node[state] (244)  at (4,6)  
	             [label={right:{\footnotesize ${244}$}}]{ };
	\node[state] (444)  at (3,7) 
	             [label={above:{\footnotesize ${444}$}}]{ };
	\node[state] (441)  at (2,6) 
	             [label={left:{\footnotesize ${441}$}}]{ };
	\node[state] (411)  at (2,5) 
	             [label={left:{\footnotesize ${411}$}}]{ };            
	             
	\path (112) edge node { } (122) edge node { } (123) edge node { } (124)
	      (122) edge node { } (221)
	      (221) edge node { } (211)
	      (211) edge node { } (111)
	      (111) edge node { } (112)
	      (123) edge node { } (233)
	      (233) edge node { } (331)
	      (331) edge node { } (311)
	      (311) edge node { } (112)
	      (124) edge node { } (244)
	      (244) edge node { } (444)
	      (444) edge node { } (441)
	      (441) edge node { } (411)
	      (411) edge node { } (112)
	      ;
\end{tikzpicture} 
\vspace{-2mm}
\caption{ }
\end{subfigure}   
\begin{subfigure}{.5\textwidth}
\centering
\begin{tikzpicture}[->,>=stealth',shorten >=1pt,auto, semithick,scale=0.7]
\tikzstyle{every state}=[draw,circle,thick, fill=black, minimum size=5pt,
                            inner sep=0pt]
	\node[state] (1221)  at (0,1)   
	             [label={left:{\footnotesize ${1221}$}}]{ };
	\node[state] (2211)  at (1,0)  
	             [label={below:{\footnotesize ${2211}$}}]{ };
	\node[state] (2111)  at (2,1) 
	             [label={below:{\footnotesize ${\text{ }\;\;\;2111}$}}]{ };
	\node[state] (1112)  at (3,3) 
	             [label={right:{\footnotesize ${1112}$}},yshift=2mm]{ };
	\node[state] (1122)  at (1,2) 
	             [label={above:{\footnotesize ${1122\;\;\;\text{ }}$}}]{ };
	
	\node[state] (1123)  at (4,1)   
	             [label={below:{\footnotesize ${1123\;\;\;\text{ }}$}}]{ };
	\node[state] (1233)  at (5,0)  
	             [label={below:{\footnotesize ${1233}$}}]{ };
	\node[state] (2331)  at (6,0)  
	             [label={right:{\footnotesize ${2331}$}}]{ };             
	\node[state] (3311)  at (6,1) 
	             [label={right:{\footnotesize ${3311}$}}]{ };
	\node[state] (3111)  at (5,2) 
	             [label={above:{\footnotesize ${\text{ }\;\;\;3111}$}}]{ };             
	             
    \node[state] (1124)  at (4,5)   
	             [label={right:{\footnotesize ${1124}$}}]{ };
	\node[state] (1244)  at (4,6)  
	             [label={right:{\footnotesize ${1244}$}}]{ };
	\node[state] (2444)  at (3.5,7) 
	             [label={right:{\footnotesize ${2444}$}}]{ };
	\node[state] (4441)  at (2.5,7) 
	             [label={left:{\footnotesize ${4441}$}}]{ };	
	\node[state] (4411)  at (2,6) 
	             [label={left:{\footnotesize ${4411}$}}]{ };
	\node[state] (4111)  at (2,5) 
	             [label={left:{\footnotesize ${4111}$}}]{ };

	\path (1112) edge node { } (1122) edge node { } (1123) edge node { } (1124)
		  (1122) edge node { } (1221)
		  (1221) edge node { } (2211)
		  (2211) edge node { } (2111)
		  (2111) edge node { } (1112)
		  (1123) edge node { } (1233)
		  (1233) edge node { } (2331)
		  (2331) edge node { } (3311)
		  (3311) edge node { } (3111)
		  (3111) edge node { } (1112)
		  (1124) edge node { } (1244)
		  (1244) edge node { } (2444)
		  (2444) edge node { } (4441)
		  (4441) edge node { } (4411)
		  (4411) edge node { } (4111)
		  (4111) edge node { } (1112)
		  ;

\end{tikzpicture} 
\vspace{-2mm}
\caption{ }
\end{subfigure}   
\vspace{-2mm}
\caption{\small (a) A Quasi-$(4,3)$-labeling of $C_5\cdot C_5\cdot C_6$ and (b) A Quasi-$(4,4)$-labeling of $C_5\cdot C_6\cdot C_7$.}
\label{fig: propel}
\end{figure}
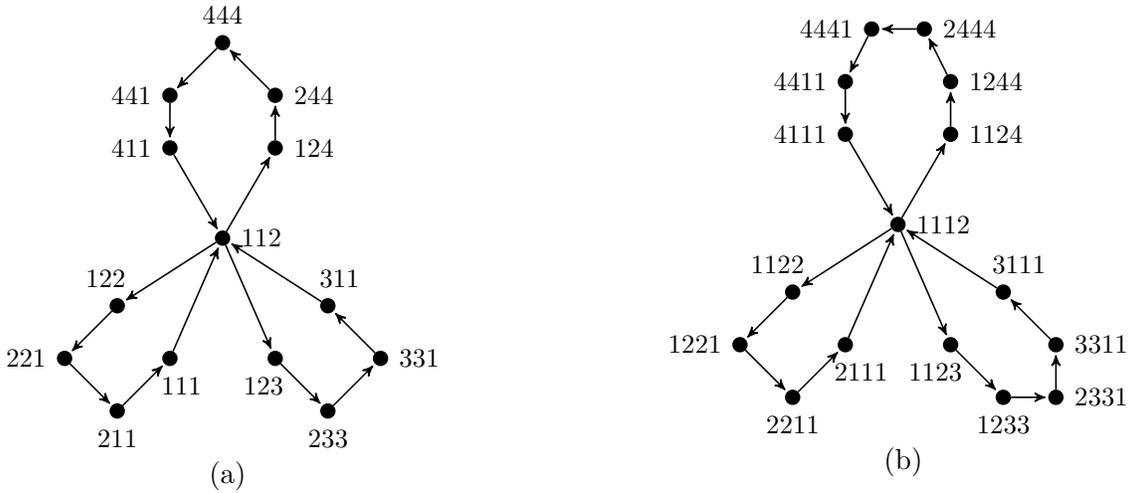

\n In Figure \ref{fig: propel}(a) we use labeling $l(v_5^1),l(v_5^2),l(v_6^3)$ from Theorem \ref{th: windmill} for $C_5,C_5,C_6$, respectively.\\
\n In Figure \ref{fig: propel}(b) we use labeling $l(v_7^3)$ from Theorem \ref{th: windmill} for $C_7$ and labeling $l^*(v_5^1)$ and $l^*(v_5^2)$ from Theorem \ref{th: propeller} for $C_5$ and $C_6$, respectively.

\pagebreak

\end{document}